\documentclass[12pt,a4paper]{article}
\usepackage[utf8]{inputenc}
\usepackage[T1]{fontenc}
\usepackage{amsmath}
\usepackage{amsthm}
\usepackage{amsfonts}
\usepackage{amssymb}
\usepackage{graphicx}
\usepackage{authblk}
\usepackage{hyperref}
\usepackage{tikz}
\usepackage{wrapfig}
\usepackage{tikz-cd}
\usepackage{caption}
\usepackage{subcaption}
\usepackage[capposition=bottom]{floatrow}
\usepackage[document]{ragged2e}
\usepackage[left=2.50cm, right=2.50cm, top=2.00cm, bottom=2.00cm]{geometry}
\usepackage{ytableau}

\newtheorem{theorem}{Theorem}[section]
\newtheorem{lemma}[theorem]{Lemma}

\newtheorem{prop}[theorem]{Proposition}
\newtheorem{corollary}[theorem]{Corollary}

\newtheorem{conjecture}[theorem]{Conjecture}

\newtheorem{innercustomthm}{Theorem}
\newenvironment{repthm}[1]
  {\renewcommand\theinnercustomthm{#1}\innercustomthm}
  {\endinnercustomthm}

\newtheorem{innercustomprop}{Proposition}
\newenvironment{repprop}[1]
  {\renewcommand\theinnercustomprop{#1}\innercustomprop}
  {\endinnercustomthm}

\theoremstyle{remark}
\newtheorem*{remark}{Remark}

\theoremstyle{definition}
\newtheorem{definition}[theorem]{Definition}

\newtheorem{innercustomdef}{Definition}
\newenvironment{repdef}[1]
  {\renewcommand\theinnercustomdef{#1}\innercustomdef}
  {\endinnercustomthm}

\title{\textbf{The Cactus Group Property for Ordinal Sums of Disjoint Unions of Chains}}
\author{Son Nguyen}
\affil{University of Minnesota - Twin Cities}
\date{}

\begin{document}
	
    \maketitle

    \begin{abstract}
    \justifying
        We study the action of Bender-Knuth involutions on linear extensions of posets and identify LE-cactus posets, i.e. those for which the cactus relations hold. It was conjectured in \cite{chiang2023bender} that d-complete posets are LE-cactus. Among the non-d-complete posets that are LE-cactus, one notable family is ordinal sums of antichains. In this paper, we characterize the LE-cactus posets in a more general family, namely ordinal sums of disjoint unions of chains.
    \end{abstract}

    \tableofcontents

\section{Introduction}\label{sec:intro}

    \justifying
    First introduced by Bender and Knuth in their study of enumerations of plane partitions and Schur polynomials \cite{benderknuth1972}, the {\it Bender--Knuth (BK) involutions}, a certain family of involutions on the set of column-strict (semi-standard) tableaux, have seen a wide range of applications across different areas of combinatorics. A classic application of BK involutions is on column-strict tableaux where they are used to prove that Schur polynomials are symmetric. Informally, the BK involutions $t_i$ act on a column-strict tableau by fixing an $i$ (resp. $i+1$) when there is an $i+1$ below (resp. $i$ above), and then swapping the contents of the remaining numbers $i$ and $i+1$ in each row.

    A {\it linear extension} of a poset $P$ is a linear order $f$ that is compatible with $P$, that is, a bijective labeling $f: P \rightarrow \{1,2,\ldots,|P|\}$ such that if $a<_P b$ then $f(a)<f(b)$. In \cite{Stanley-promotion-evacuation}, Stanley introduced an analog of the BK involutions $t_i$ on linear extensions of a poset $P$, which swaps two adjacent labels $i$ and $i+1$ when they label incomparable elements of $P$ and fixes them otherwise.

    In this paper, we study a family of relations among the BK involutions called {\it cactus relations}, which present the cactus group. For any poset $P$ with $|P| = n$, it is easy to see that the BK involutions $t_1,\ldots,t_{n-1}$ acting on the linear extensions of $P$ satisfy the  relations $t_i^2=1$ and $t_it_j=t_jt_i$ for $|i-j| \geq 2$. On the other hand, for some posets $P$ (discussed extensively in \cite{chiang2023bender}), they {\it fail} to satisfy the extra family of relations that define the {\it cactus group}, namely $(t_iq_{jk})^2 = 1$ for $i+1<j<k$ where $q_{jk}: = q_{k-1}q_{k-j}q_{k-1}$ and $q_{i}:= t_1(t_2t_1)\ldots(t_it_{i-1}\ldots t_1)$.  See Section~\ref{subsec:cactus-relation} for a fuller discussion of these relations. In \cite{chiang2023bender}, a poset $P$ was called {\it LE-cactus} whenever the cactus group relations hold when $\{t_i\}$ act on its linear extensions. 
    
    Our main concern is the following question: for which posets do the BK involutions satisfy the cactus relations (defined in Definition \ref{def:LE-cactus})? The authors of \cite{CGP} showed that Ferrers posets are LE-cactus. In \cite{chiang2023bender}, several other families of LE-cactus posets were found, such as shifted Ferrers posets, rooted trees, and other minuscule posets. The paper also made the following conjecture about \textit{d-complete} posets. We refer the readers to \cite{proctor2019d} and \cite{kim2019hook} for the precise definition of d-complete posets.

    \begin{conjecture}[{\cite[Conjecture 3.23]{chiang2023bender}}]\label{con:d-complete}
        d-complete posets are LE-cactus.
    \end{conjecture}
    
    However, a large number of LE-cactus posets remained uncategorized. For example, a larger family of posets that includes all d-complete posets is the \textit{jeu-de-taquin} posets. \textit{Most} (but not all) jeu-de-taquin posets are LE-cactus: among all jeu-de-taquin posets of size up to 9, only 1 is not LE-cactus. Furthermore, there are many other posets that are not jeu-de-taquin but are also LE-cactus; one notable family involves ordinal sums of antichains. In this paper, we characterize the LE-cactus posets in a more general family, namely ordinal sums of disjoint unions of chains. Since most posets in this family are not LE-cactus, our result is a complement of Conjecture \ref{con:d-complete}.

    Furthermore, in \cite{chiang2023bender}, the following result about disjoint unions of LE-cactus posets was proved.

    \begin{theorem}[{\cite[Theorem 3.17]{chiang2023bender}}]\label{thm:P+Q-cactus}
        If $P$ and $Q$ are LE-cactus, then their disjoint union, $P+Q$, is LE-cactus.
    \end{theorem}

    On the other hand, little is known about ordinal sums of LE-cactus posets, i.e. if $P$ and $Q$ are LE-cactus, when is $P\oplus Q$ LE-cactus? Some progress was made in \cite{chiang2023bender}.

    \begin{prop}[{\cite[Proposition 3.18, 3.19, 3.20]{chiang2023bender}}]\label{prop:A_ioplusP}
        Let $A_m$ be the antichain of size $m$. If $P$ is LE-cactus, then $A_1\oplus P$ and $A_2\oplus P$ are LE-cactus. However, for any non-empty finite poset $P$, $A_m\oplus P$ is not LE-cactus for $m > 3$.
    \end{prop}

    Since disjoint unions of chains are LE-cactus, our main theorem about their ordinal sums is another step towards understanding the ordinal sum operation.

    Let us summarize our main result. Let $\mathfrak{D}_n$, with $n>1$, be the set of all posets with $n$ elements that are disjoint unions of at least two chains. For completeness, we define $\mathfrak{D}_1$ to include the poset $C_1$ containing one element. Let $\lambda = (\lambda_1,\ldots,\lambda_\ell)$ be a partition of $n$ with $\ell > 1$. We define $D_\lambda$ to be a disjoint union of $\ell$ chains such that the $i$th chain has $\lambda_i$ elements. 

    \begin{definition}\label{def:cactus-compatible}
        We say a triple $(p,n,q)$ is \textbf{cactus-compatible} if it satisfies the following condition: Let $P$ and $Q$ be any poset with $|P| = p$ and $|Q| = q$; let $D_\lambda$ be any poset in $\mathfrak{D}_n$. Let $R = P~\oplus~D_\lambda~\oplus~Q$, and let $f$ be any linear extension of $R$. Then for all $i+1<j<k$, the element $(t_iq_{jk})^2$ fixes the labels of the elements in $D_\lambda$ when acting on $f$.
    \end{definition}
    
    Our main theorem, Theorem \ref{thm:char}, follows the following three propositions.

    \begin{repprop}{\ref{thm:main-thm}}
        For $n>3$, $(p,n,q)$ is cactus-compatible if and only if
        \begin{enumerate}
            \item $p>q+n-4$, or
            \item $p = q+n-4$ and $q~\text{mod}~n\neq 1,3$, or
            \item $p = q+n-r$ for $r>4$ and $q~\text{mod}~ n > r-1$.
        \end{enumerate}
        In particular, if $p\leq q-1$, $(p,n,q)$ is not cactus-compatible.
    \end{repprop}

    \begin{repprop}{\ref{prop:D3}}
        $(p,3,q)$ is cactus-compatible if and only if
        \begin{enumerate}
            \item $p>q-1$, or
            \item $p = q-1$ and $q~\text{mod}~n\neq 1,3$.
        \end{enumerate}
        In particular, if $p\leq q-1$, $(p,3,q)$ is not cactus-compatible.
    \end{repprop}
        
    \begin{repprop}{\ref{prop:D1,2}}
        For all $p,q$, $(p,1,q)$ and $(p,2,q)$ are cactus-compatible.
    \end{repprop}

    Our main theorem is the following characterization.

    \begin{repthm}{\ref{thm:char}}
        Consider a sequence of positive integers $a_0=0,a_1,\ldots,a_\ell,a_{\ell+1} = 0$ and a sequence of posets $D_{\mu_1},\ldots,D_{\mu_\ell}$ where $D_{\mu_i}\in\mathfrak{D}_{a_i}$. The poset $P = D_{\mu_1}\oplus \ldots\oplus D_{\mu_\ell}$ is LE-cactus if and only if for all $i = 1,2,\ldots,\ell$, the triples
        \[ \left(\sum_{r=0}^{i-1} a_r, \,\, a_i, \,\, \sum_{r=i+1}^{\ell+1} a_r\right) \]
        are cactus-compatible.
    \end{repthm}

    The paper is outlined as follows: in Section \ref{sec:def-res}, we will review the key definitions and results on Bender-Knuth involutions, promotion and evacuation, and cactus relations. Then in Section \ref{sec:pro-eva}, we will study the actions of promotion and evacuation on linear extensions of $D_\lambda$. We will also introduce a way to think about these actions as permuting numbers on ordered set partitions. Finally, in Section \ref{sec:main-thm}, we will prove our main theorem through a few results.

    \begin{remark}
        We can think of a linear extension of a poset $P$ as a bijection from the elements of $P$ to the set $\{1,\ldots,p\}$ where $p = |P|$. Hence, in this paper, when we use the operation mod $p$, we mean the result is in $\{1,\ldots,p\}$ instead of $\{0,\ldots,p-1\}$.
    \end{remark}

    \noindent\textbf{Acknowledgments} I would like to thank Vic Reiner for introducing the subject to me, and for his extremely valuable guidance and support. I would like to thank the 2022 University of Minnesota Combinatorics and Algebra REU, supported by RTG grant NSF/DMS-1745638, for organizing the program in which this project was initiated. I thank Judy Chiang, Anh Hoang, Matthew Kendall, Ryan Lynch, Benjamin Przybocki, Janabel Xia, Pasha Pylyavksyy, and Sylvester Zhang for their helpful discussions and comments. I thank Connor McCausland for their help with proofreading.

\section{Definitions and Results}\label{sec:def-res}

    \subsection{Poset operations}\label{subsec:poset-ops}

    Let $P$ and $Q$ be any finite posets with the partial orders $\leq_P$ and $\leq_Q$ respectively. The \textit{ordinal sum} of $P$ and $Q$ is the poset $R$ whose elements are those in $P \cup Q$, and $a\leq_R b$ if and only if
    \begin{itemize}
        \item $a,b\in P$ and $a\leq_P b$, or
        \item $a,b\in Q$ and $a\leq_Q b$, or
        \item $a\in P$ and $b\in Q$.
    \end{itemize}
    We denote the ordinal sum of $P$ and $Q$ as $P~\oplus~Q$. On the other hand, the \textit{disjoint union} of $P$ and $Q$ is the poset $R$ whose elements are those in $P \cup Q$, and $a\leq_R b$ if and only if
    \begin{itemize}
        \item $a,b\in P$ and $a\leq_P b$, or
        \item $a,b\in Q$ and $a\leq_Q b$.
    \end{itemize}
    We denote the disjoint union of $P$ and $Q$ as $P+Q$. Finally, a special family of posets that we will consider is the chain posets in which the partial order is a total order. We denote the chain poset with $n$ elements as $C_n$.

    \subsection{Bender-Knuth involutions}\label{subsec:BK-involutions}

    A {\it linear extension} of a finite poset $P$ is a linear order $f$ that is compatible with $P$, that is, a bijective labeling $f: P \rightarrow \{1,2,\ldots,|P|\}$ such that if $a<_P b$ then $f(a)<f(b)$. Let $\mathcal{L}(P)$ be the set of linear extensions of a poset $P$. The BK involutions act on linear extensions of any poset $P$ as follows: each $t_i:\mathcal{L}(P)\rightarrow\mathcal{L}(P)$ is a bijection that swaps two adjacent labels $i$ and $i+1$ when they label incomparable elements of $P$ and fixes them otherwise.
    
    Let us briefly point out that the BK involutions defined here are motivated by the classical BK involutions on {\it Young tableaux}. Recall that given a partition $\lambda = (\lambda_1,\ldots,\lambda_\ell)$ of $n$, the \textit{Young diagram} of $\lambda$ is a collection of $n$ left-justified boxes in $\ell$ rows such that row $i$ has $\lambda_i$ boxes. A \textit{standard Young tableau} of shape $\lambda$ is a filling of the Young diagram of $\lambda$ with the numbers $1,2,\ldots,n$ such that the numbers strictly increase from left-to-right in each row and from top-to-bottom in each column.
    
    The \textit{Ferrers} poset $F_\lambda$ of $\lambda$ is the set $\{(i,j)~|~1\leq j \leq \lambda_i\}$ with the partial order $(i,j)<(i,j+1)$ and $(i,j)<(i+1,j)$. Observe that every linear extension of a Ferrers poset can be viewed as a standard Young tableau, as shown in Figure \ref{fig:SYT-Ferrer}. Thus, in the special case of Ferrers posets, the $t_i$ defined above can be identified with a special case of the classical BK involutions on standard Young tableaux. We refer the readers to \cite{chiang2023bender} for a more thorough discussion of these classical BK involutions.
    
    \begin{figure}[h!]
        \centering
        \includegraphics[width = 0.4\textwidth]{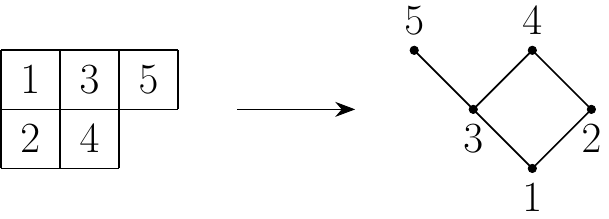}
        \caption{Standard Young tableau and Ferrers poset}
        \label{fig:SYT-Ferrer}
    \end{figure}

    \subsection{Cactus relations}\label{subsec:cactus-relation}    

    The terminology {\it cactus group} was introduced in work of Henriques and Kamnitzer \cite{henriques2006crystals}, as a name 
    for the fundamental group of the moduli space $\overline{M}_{0,n+1}(\mathbb{R})$ of real genus zero stable curves with $n+1$ marked points,
    appearing in work of Devadoss \cite{devadoss239tessellations}
    and Davis, Januszkiewicz and Scott \cite{davis2003fundamental}.

    \begin{definition}[Cactus group]\label{def:cactus-group}
        The cactus group $\mathcal{C}_n$ is generated by $q_{[i,j]}, 1\leq i < j \leq n$, satisfying the relations
        \begin{enumerate}
            \item $q_{[i,j]}^2 = 1$,
            \item $q_{[i,j]}q_{[k,l]}=q_{[k,l]}q_{[i,j]}$ if $j<k$,
            \item $q_{[i,j]}q_{[k,l]}q_{[i,j]}=q_{[i+j-l,i+j-k]}$ if $i\leq k<l\leq j$.
        \end{enumerate}
    \end{definition}

    We note that there is a well-defined group homomorphism from the cactus group $\mathcal{C}_n$ to the symmertic group $\mathfrak{S}_n$, sending $q_{[i,j]}$ to $\left(\begin{smallmatrix} 
    i & i+1 & \cdots & j\\
    j & j-1 & \cdots & i
    \end{smallmatrix}\right)$. For example, letting $(i,j,k,l) = (2,7,3,5)$, the third relation in Definition \ref{def:cactus-group} becomes $q_{[2,7]}q_{[3,5]}q_{[2,7]} = q_{[4,6]}$. Indeed,
    \[ 12345678 \xrightarrow{q_{[2,7]}} 17654328 \xrightarrow{q_{[3,5]}} 17456328 \xrightarrow{q_{[2,7]}} 12365478 \]
    and
    \[ 12345678 \xrightarrow{q_{[4,6]}} 12365478. \]
    The reason why we want to note this example is because we will see a similar idea in the proof of our main theorem. In \cite{CGP}, Chmutov, Glick and Pylyavskyy proved that the action of BK involutions on column-strict tableaux satisfies the cactus relation by introducing an isomorphic presentation of the cactus group that will be used in this paper.

    \begin{theorem}[{\cite[Theorem 1.8]{CGP}}]\label{thm:cactus-presentation}
        The relations in Definition \ref{def:cactus-group} for $\mathcal{C}_n$ are equivalent to the following relations on the generators $t_i, i = 1,\ldots,n-1$:
        \begin{enumerate}
            \item $t_i^2 = 1$,
            \item $(t_tt_j)^2 = 1$ if $|i-j|>1$,
            \item $(t_iq_{k-1}q_{k-j}q_{k-1})^2 = 1$ if $i+1<j<k$,
        \end{enumerate}
        where we define
        \begin{equation}
        \label{qi-definition}
        q_{i} = t_1(t_2t_1)\ldots(t_it_{i-1}\ldots t_1).
        \end{equation}
        For our convenience, we also define 
        \begin{equation}
        \label{qij-definition}
        q_{jk} = q_{k-1}q_{k-j}q_{k-1}
        \end{equation}
        so that the third relation becomes $(t_iq_{jk})^2 = 1$ if $i+1<j<k$.
    \end{theorem}

    On linear extensions of posets, the first two relations in Theorem \ref{thm:cactus-presentation} always hold. The last relation, $(t_iq_{jk})^2 = 1$ if $i+1<j<k$, does not. For example, one can check that the relation $(t_1q_{34})^2 = 1$ does not hold for the linear extension in Figure \ref{fig:min-non-cactus}. This motivates the following definition.

    \begin{definition}[{\cite[Definition 3.10]{chiang2023bender}}]\label{def:LE-cactus}
        Call the relation $(t_iq_{jk})^2 = 1$ if $i+1<j<k$ the \textbf{cactus relation}, and call posets on which this relation holds \textbf{LE-cactus} posets.
    \end{definition}

    \begin{figure}[h!]
        \centering
        \includegraphics[scale = 0.7]{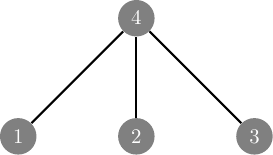}
        \caption{A non-LE-cactus poset}
        \label{fig:min-non-cactus}
    \end{figure}

    \subsection{Promotion and Evacuation}\label{subsec:pro-and-eva}

    In \cite{Stanley-promotion-evacuation}, Stanley gave two operations on linear extensions: promotion and evacuation. These will be extremely useful for understanding our main result. First, we introduce \textit{promotion} $\partial_i$.

    \begin{definition}[Promotion]\label{def:promotion}
        Let $1\leq i\leq |P|$, promotion $\partial_i:\mathcal{L}(P)\rightarrow\mathcal{L}(P)$ is a bijection that sends a linear extension $f$ of $P$ to $f' = \partial_i(f)$ by the following procedure:
        \begin{enumerate}
            \item Let $t_1\in P$ satisfy $f(t_1) = 1$ and remove the label $1$ from $t_1$.
            \item Among the elements of $P$ covering $t_1$, let $t_2$ be the one with the smallest label $f(t_2)$, "slide" this label down to $t_1$, i.e. remove this label from $t_2$ and place it at $t_1$.
            \item Repeat the procedure until reaching an element $t_k$ such that either $t_k$ is a maximal element or no element covering $t_k$ has label less than or equal to $i$.
            \item Label $t_k$ with $i+1$ and decrease every label from $2$ to $i+1$ by one. Note that at this point there might be two labels $i+1$, we only decrease the one that labels $t_k$.
        \end{enumerate}
    \end{definition}

    The (saturated) chain $t_1<t_2<\ldots<t_k$ above is called the \textit{promotion chain}. Figure \ref{fig:ex-promotion} shows an example of promotion $\partial_5$ of the linear extension $f$ shown in Figure \ref{subfig:ex-promotion1}. The red elements in Figure \ref{subfig:ex-promotion6} form the promotion chain. Note that before the final step (Figure \ref{subfig:ex-promotion5}), there are two labels 6. In the final step, we only decrease the one labeling the top element of the promotion chain (the label 6 in red).

    \begin{figure}[h!]
    		
     \centering
        \begin{subfigure}[b]{0.25\textwidth}
            \centering
            \includegraphics[width = 0.8\textwidth]{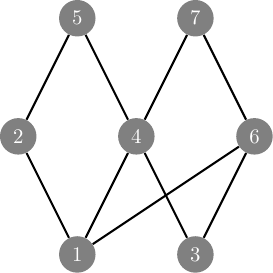}
            \caption{Linear extension $f$}
            \label{subfig:ex-promotion1}
        \end{subfigure}
     \quad\quad
        \begin{subfigure}[b]{0.25\textwidth}
            \centering
            \includegraphics[width = 0.8\textwidth]{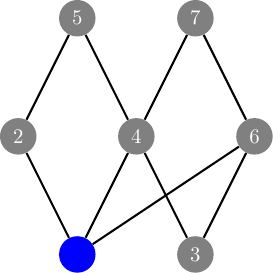}
            \caption{Remove label 1}
            \label{subfig:ex-promotion2}
        \end{subfigure}
     \quad\quad
        \begin{subfigure}[b]{0.25\textwidth}
            \centering
            \includegraphics[width = 0.8\textwidth]{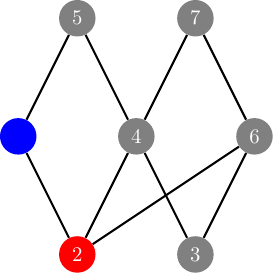}
            \caption{Slide label 2 down}
            \label{subfig:ex-promotion3}
        \end{subfigure}

        \vspace{1em}
    
        \begin{subfigure}[b]{0.25\textwidth}
            \centering
            \includegraphics[width = 0.8\textwidth]{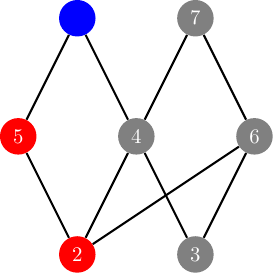}
            \caption{Slide label 5 down}
            \label{subfig:ex-promotion4}
        \end{subfigure}
     \quad\quad
        \begin{subfigure}[b]{0.25\textwidth}
            \centering
            \includegraphics[width = 0.8\textwidth]{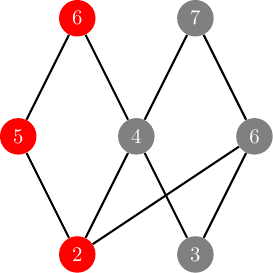}
            \caption{Add label 6}
            \label{subfig:ex-promotion5}
        \end{subfigure}
     \quad\quad
        \begin{subfigure}[b]{0.25\textwidth}
            \centering
            \includegraphics[width = 0.8\textwidth]{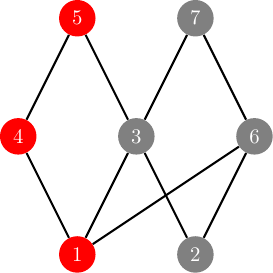}
            \caption{Decrease by one}
            \label{subfig:ex-promotion6}
        \end{subfigure}
        
        \caption{Promotion $\partial_5$}
        \label{fig:ex-promotion}
        
    \end{figure}

    While it is not clear from the definition of promotions that they are related to Bender-Knuth moves $t_i$, Stanley showed that they are indeed the combinatorial interpretation for $t_{i-1}t_{i-2}\ldots t_1$.

    \begin{theorem}[{\cite{Stanley-promotion-evacuation}}]
        For any poset $P$, $\partial_i = t_{i-1}t_{i-2}\ldots t_1$ for $1<i\leq |P|$.
    \end{theorem}

    \begin{figure}[h!]
    		
     \centering
        \begin{subfigure}[b]{0.25\textwidth}
            \centering
            \includegraphics[width = 0.8\textwidth]{Fig/ex-promotion/ex-promotion1.pdf}
            \caption{Linear extension $f$}
            \label{subfig:ex-pro-start}
        \end{subfigure}
     \quad\quad
        \begin{subfigure}[b]{0.25\textwidth}
            \centering
            \includegraphics[width = 0.8\textwidth]{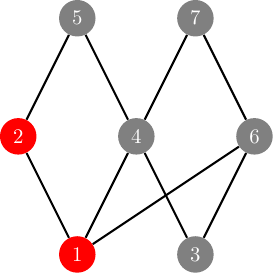}
            \caption{$t_1$}
            \label{subfig:ex-pro-t1}
        \end{subfigure}
     \quad\quad
        \begin{subfigure}[b]{0.25\textwidth}
            \centering
            \includegraphics[width = 0.8\textwidth]{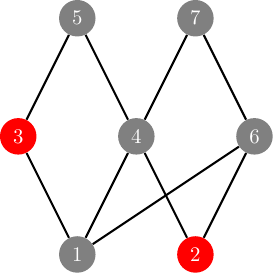}
            \caption{$t_2$}
            \label{subfig:ex-pro-t2}
        \end{subfigure}

        \vspace{1em}
    
        \begin{subfigure}[b]{0.25\textwidth}
            \centering
            \includegraphics[width = 0.8\textwidth]{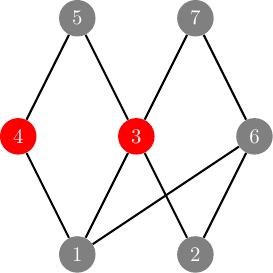}
            \caption{$t_3$}
            \label{subfig:ex-pro-t3}
        \end{subfigure}
     \quad\quad
        \begin{subfigure}[b]{0.25\textwidth}
            \centering
            \includegraphics[width = 0.8\textwidth]{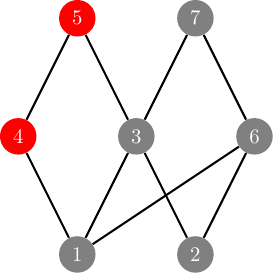}
            \caption{$t_4$}
            \label{subfig:ex-pro-t4}
        \end{subfigure}
     \quad\quad
        \begin{subfigure}[b]{0.25\textwidth}
            \centering
            \includegraphics[width = 0.8\textwidth]{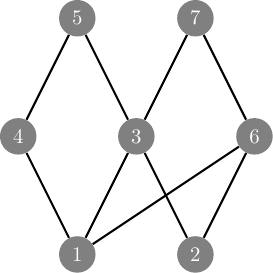}
            \caption{Result}
            \label{subfig:ex-pro-res}
        \end{subfigure}
        
        \caption{The action of $t_4t_3t_2t_1$}
        \label{fig:ex-pro-t}
        
    \end{figure}

    For example, Figure \ref{fig:ex-pro-t} shows the action of $t_4t_3t_2t_1$ on the same linear extension as in Figure \ref{subfig:ex-promotion1}, and the result in Figure \ref{subfig:ex-pro-res} is the same as in Figure \ref{subfig:ex-promotion6}. Thus, we can define $\partial_i = t_{i-1}t_{i-2}\ldots t_1$ for $i>1$ and $\partial_1 = 1$, and we can express the operator operator appearing in \eqref{qi-definition} as $q_{k-1} = \partial_1\partial_2\ldots\partial_k$. With the definition of promotion, the combinatorial interpretation of $q_i$, called \textit{evacuation}, can be described easily.

    \begin{definition}[Evacuation]\label{def:evacuation}
        Let $1\leq i\leq |P|$, evacuation $q_i:\mathcal{L}(P)\rightarrow\mathcal{L}(P)$ is a bijection that sends a linear extension $f$ of $P$ to $f' = q_i(f)$ by the following procedure:
        \begin{enumerate}
            \item Apply $\partial_{i+1}$ and "freeze" the label $i+1$.
            \item Apply $\partial_{i}$ and "freeze" the label $i$.
            \item Continue applying promotion until every label from $1$ to $i+1$ is frozen.
        \end{enumerate}
        We will occasionally refer to $\partial_{i+1}$ as the first round of promotion in $q_i$. Similarly, $\partial_i$ is the second round of promotion and so on. Thus, $q_i$ has $i+1$ rounds of promotion.
    \end{definition}

    \begin{figure}[h!]
    		
     \centering
        \begin{subfigure}[b]{0.25\textwidth}
            \centering
            \includegraphics[width = 0.8\textwidth]{Fig/ex-promotion/ex-promotion1.pdf}
            \caption{Linear extension $f$}
            \label{subfig:ex-eva-start}
        \end{subfigure}
     \quad\quad
        \begin{subfigure}[b]{0.25\textwidth}
            \centering
            \includegraphics[width = 0.8\textwidth]{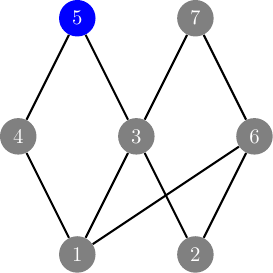}
            \caption{$\partial_5$}
            \label{subfig:ex-eva5}
        \end{subfigure}
     \quad\quad
        \begin{subfigure}[b]{0.25\textwidth}
            \centering
            \includegraphics[width = 0.8\textwidth]{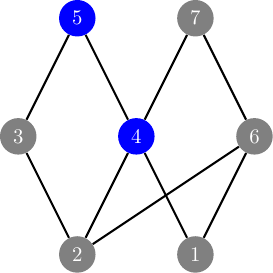}
            \caption{$\partial_4$}
            \label{subfig:ex-eva4}
        \end{subfigure}

        \vspace{1em}
    
        \begin{subfigure}[b]{0.25\textwidth}
            \centering
            \includegraphics[width = 0.8\textwidth]{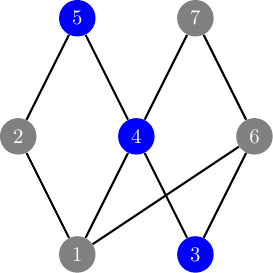}
            \caption{$\partial_3$}
            \label{subfig:ex-eva3}
        \end{subfigure}
     \quad\quad
        \begin{subfigure}[b]{0.25\textwidth}
            \centering
            \includegraphics[width = 0.8\textwidth]{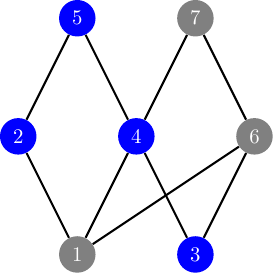}
            \caption{$\partial_2$}
            \label{subfig:ex-eva2}
        \end{subfigure}
     \quad\quad
        \begin{subfigure}[b]{0.25\textwidth}
            \centering
            \includegraphics[width = 0.8\textwidth]{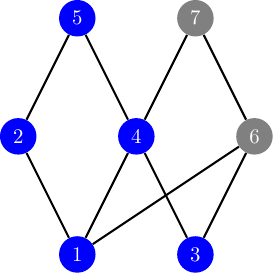}
            \caption{$\partial_1$}
            \label{subfig:ex-eva1}
        \end{subfigure}
        
        \caption{Evacuation $q_4$}
        \label{fig:ex-eva}
        
    \end{figure}

     Figure \ref{fig:ex-eva} shows an example of evacuation. We want to point out that there is an offset by $1$ in the definition of evacuation as $q_i$ includes $i+1$ promotions $\partial_{i+1}, \partial_i,\ldots,\partial_1$. This is, unfortunately, unavoidable because the action of $t_{i-1}t_{i-2}\ldots t_1$ affects the labels from 1 up to $i$. We refer the reader to \cite{Stanley-promotion-evacuation} for further details about promotion and evacuation. What we want to emphasize is the following property.

    \begin{prop}[{\cite[Theorem 2.1]{Stanley-promotion-evacuation}}]\label{prop:eva-involution}
        Evacuation is an involution, that is $q_i^2 = 1$.
    \end{prop}

    For example, we invite the readers to apply $q_4$ to the linear extension in Figure \ref{subfig:ex-eva1}. By Proposition \ref{prop:eva-involution}, one should expect the result to be the linear extension in Figure \ref{subfig:ex-eva-start}.

    \begin{corollary}\label{cor:qjk-involution}
        The operator $q_{jk}$ defined in \eqref{qij-definition} is also an involution, that is $q_{jk}^2 = 1$.
    \end{corollary}

    Thus, the cactus relation $(t_iq_{jk})^2 = 1$ is equivalent to $t_i$ and $q_{jk}$ commuting.

\subsection{Unions of Antichains}\label{subsec:D_n_necklace}

    Now we introduce our main concern of this paper.

    \begin{definition}\label{def:D_n}
        For $n>1$, we define $\mathfrak{D}_n$ to be the set of all posets with $n$ elements that are unions of at least two chains. For example, Figure \ref{fig:D_3_posets} shows the two posets in $\mathfrak{D}_3$. For completeness, we define $\mathfrak{D}_1$ to include the poset $C_1$ containing one element.
    \end{definition}
        
    \begin{remark}
        If a poset is a single chain, all Bender-Knuth moves become the identity, which is not interesting. Thus, we want at least two chains in our definition.
    \end{remark}

    \begin{figure}[h!]
        \centering
        \includegraphics[width = 0.5\textwidth]{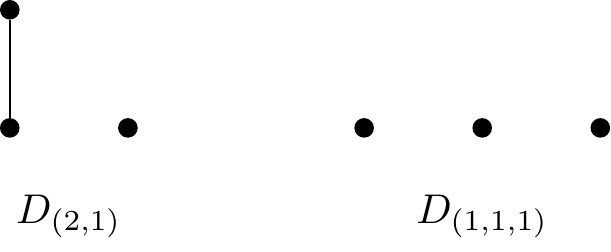}
        \caption{$\mathfrak{D}_3$}
        \label{fig:D_3_posets}
    \end{figure}

    \begin{definition}\label{def:D_lamda}
        Let $\lambda = (\lambda_1,\ldots,\lambda_\ell)$ be a partition of $n$ with $\ell > 1$. We define $D_\lambda$ to be a union of $\ell$ chains such that the $i$th chain has $\lambda_i$ elements, that is, $D_\lambda = C_{\lambda_1} + C_{\lambda_2} + \ldots + C_{\lambda_\ell}$.
    \end{definition}

    \begin{remark}
         The order of the chains does not matter, so we can assume that the lengths of the chains are decreasing and associate each poset in $\mathfrak{D}_n$ with a partition of $n$. In other words, $\mathfrak{D}_n = \{D_\lambda~|~\lambda\vdash n,~ \ell(\lambda) > 1\}$ for $n>1$, and $\mathfrak{D}_1 = \{D_{(1)}\}$.
    \end{remark}

    Observe that a linear extension $f$ of $D_\lambda=C_{\lambda_1}+\cdots+C_{\lambda_\ell}$ is completely determined by the {\it ordered set partition} $(L_1,\ldots,L_{\lambda_\ell})$  of the set $\{1,2,\ldots,n\}$ having blocks $L_i:=f^{-1}(C_i)$.

\section{Promotion and Evacuation}\label{sec:pro-eva}

    In this section, we will analyze the behavior of linear extensions of disjoint union of chains under promotion and evacuation. First, we study the effect of promotion on linear extensions of $D_\lambda$ on the corresponding ordered set partition.

    \begin{lemma}\label{lem:D_n-promotion}
        Let $f$ be a linear extension of $D_\lambda\in\mathfrak{D}_n$. If a chain in $D_\lambda$ contains the label $i$, then after $\partial_k$, this chain contains the label $i$ if $i>k$, and it contains the label $i-1~(\text{mod}~ k)$ if $i\leq k$. In particular, this chain contains the label $k$ if $i = 1$.
    \end{lemma}

    \begin{proof}
        This is apparent from the definition of promotion. Since $\partial_k$ does not affect any label greater than $k$, if $i>k$, this label stays at the same element after $\partial_k$. If $1<i\leq k$, then this label stays in the same chain during promotion because the chains are disjoint, and it is decreased by one in the final step of promotion. Hence, the chain that originally contained $i$ will contain $i-1$ after promotion $\partial_k$. Finally, if $i = 1$, then this label is removed at the beginning of promotion. However, because the chains are disjoint, the final element of the promotion chain is in the same chain. This element is eventually labeled $k+1$ and then decreased to $k$. Thus, the chain that originally contains $1$ will contain $k$ after promotion $\partial_k$.
    \end{proof}

    In terms of ordered set partitions, the action of $\partial_k$ can be described as follows: keep the numbers from $k+1$ to $n$ in the same blocks. Then, send $1$ to the block containing $2$, send $2$ to the block containing $3$, $\ldots$, and send $k$ to the block containing $1$. In other words, promotion is left multiplication by $(k\ldots 21)$. Figure \ref{fig:pro-neck} shows an example of the correspondence between applying $\partial_5$ to the linear extension and applying $(54321)$ to the ordered set partition.

    \begin{figure}[h!]
        \centering
        \includegraphics[width = 0.7\textwidth]{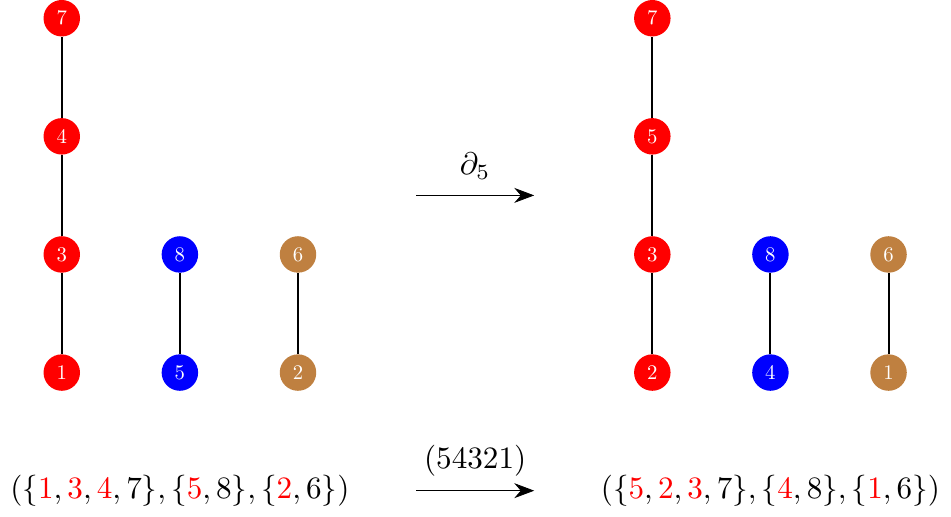}
        \caption{Promotion $\partial_5$ on a linear extension and the corresponding ordered set partition}
        \label{fig:pro-neck}
    \end{figure}

    Next, we will look at the effect of evacuation.

    \begin{lemma}\label{lem:D_n-evacuation}
        Let $f$ be a linear extension of $D_\lambda\in\mathfrak{D}_n$. If a chain in $D_\lambda$ contains the label $i$, then after $q_{k-1}$, this chain contains the label $i$ if $i>k$ and $k+1-i$ if $i\leq k$.
    \end{lemma}

    \begin{proof}
        Recall that $q_{k-1}$ actually includes $k$ rounds of promotion: $\partial_k,\partial_{k-1},\ldots,\partial_1$. Since $q_{k-1}$ does not affect any label greater than $k$, if $i>k$, this label stays at the same element after $q_{k-1}$. If $1\leq i\leq k$, let $P_j$ be the chain in $D_\lambda$ that originally contained the label $i$. By Lemma \ref{lem:D_n-promotion}, after the first $i-1$ rounds of promotion $\partial_k,\partial_{k-1},\ldots,\partial_{k-i+2}$, $P_j$ contains the label $1$. Thus, after applying $\partial_{k-i+1}$ in the $i$th round of promotion, $P_j$ contains the label $k-i+1$. None of the subsequent rounds of promotion affect the label $k-i+1$, so this label stays in $P_j$. This completes the proof.
    \end{proof}

    Similar to promotion, in terms of ordered set partitions, the action of $q_{k-1}$ can be described as follows: keep the numbers from $k+1$ to $n$ in the same blocks. Then, send $1$ to the block containing $k$, send $2$ to the block containing $k-1$, $\ldots$, and send $k$ to the block containing $1$. In other words, evacuation is left multiplication by $\left(\begin{smallmatrix} 
    1 & 2 & \cdots & k\\
    k & k-1 & \cdots & 1
    \end{smallmatrix}\right)$. Figure \ref{fig:eva-neck} shows an example of the correspondence between applying $q_4$ to the linear extension and applying $\left(\begin{smallmatrix} 
    1 & 2 & 3 & 4 & 5\\
    5 & 4 & 3 & 2 & 1
    \end{smallmatrix}\right)$ to the ordered set partition.

    \begin{figure}[h!]
        \centering
        \includegraphics[width = 0.7\textwidth]{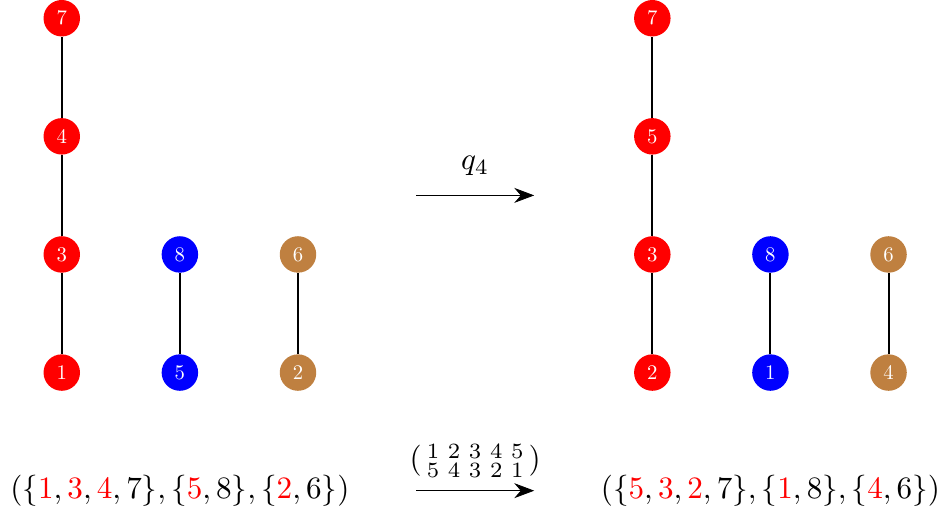}
        \caption{Evacuation $q_4$ on a linear extension and the corresponding ordered set partition}
        \label{fig:eva-neck}
    \end{figure}

\section{Proof of main theorem}\label{sec:main-thm}

    In this section, we will prove the main theorem. First of all, our main concern will be posets $R = P~\oplus~D_\lambda~\oplus~Q$ where $D_\lambda\in \mathfrak{D}_n$, and $P$ and $Q$ are any finite posets with $|P| = p$ and $|Q| = q$. Since $R$ is an ordinal sum, for any linear extension $f$ of $R$, $f^{-1}(P) = \{1,\ldots,p\}$, $f^{-1}(D_\lambda) = \{p+1,\ldots,p+n\}$, and $f^{-1}(Q) = \{p+n+1,\ldots,p+n+q\}$. Thus, a linear extension $f$ of $R$ induces a linear extension $g$ of $D_\lambda$ defined by $g(x) = f(x)-p$ for $x\in D_\lambda$.
    
    \begin{definition}\label{def:induced_map}
        We call the induced map defined above $I$. 
    \end{definition}

    \begin{figure}[h!]
        \centering
        \includegraphics[width = 0.7\textwidth]{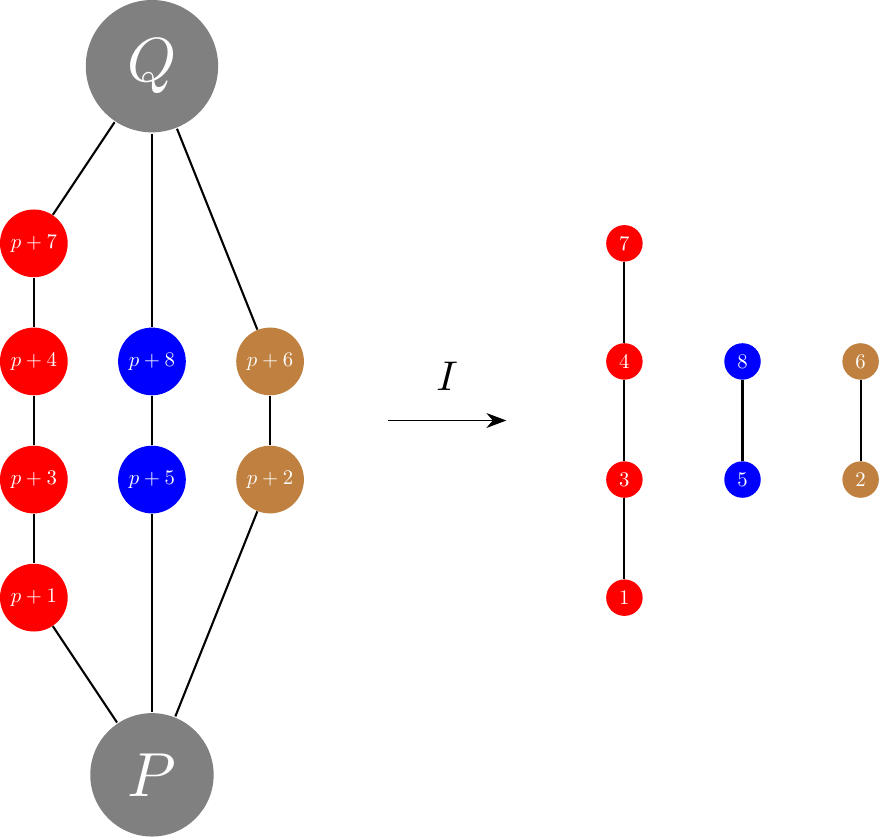}
        \caption{An example of the induced map $I$}
        \label{fig:I-ex}
    \end{figure}

    Figure \ref{fig:I-ex} gives an example of this map. The goal of the main theorem is to characterize the conditions on $p$ and $q$ such that $(t_iq_{jk})^2$ fixes the labels of the elements in $D_\lambda$ for all $i+1<j<k$, that is, $I((t_iq_{jk})^2(f)) = I(f)$ for all $i+1<j<k$. Recall from Section \ref{sec:intro} the following definition.

    \begin{repdef}{\ref{def:cactus-compatible}}
        We say a triple $(p,n,q)$ is \textbf{cactus-compatible} if it satisfies the following condition: Let $P$ and $Q$ be any poset with $|P| = p$ and $|Q| = q$; let $D_\lambda$ be any poset in $\mathfrak{D}_n$. Let $R = P~\oplus~D_\lambda~\oplus~Q$, and let $f$ be any linear extension of $R$. Then for all $i+1<j<k$, the element $(t_iq_{jk})^2$ fixes the labels of the elements in $D_\lambda$ when acting on $f$.
    \end{repdef}
    
    First, we prove the following proposition.

    \begin{prop}\label{thm:main-thm}
        For $n>3$, $(p,n,q)$ is cactus-compatible if and only if
        \begin{enumerate}
            \item $p>q+n-4$, or
            \item $p = q+n-4$ and $q~\text{mod}~n\neq 1,3$, or
            \item $p = q+n-r$ for $r>4$ and $q~\text{mod}~ n > r-1$.
        \end{enumerate}
        In particular, if $p\leq q-1$, $(p,n,q)$ is not cactus-compatible.
    \end{prop}

    We will need a few lemmas to prove this proposition.

    \begin{lemma}\label{lem:t_i-fixes}
        The only moves that affect the labels in $D_\lambda$ are $t_{p+1},\ldots,t_{p+n-1}$.
    \end{lemma}

    \begin{proof}
        Since $R$ is an ordinal sum, for any linear extension $f$ of $R$, $f^{-1}(P) = \{1,\ldots,p\}$, $f^{-1}(D_\lambda) = \{p+1,\ldots,p+n\}$, and $f^{-1}(Q) = \{p+n+1,\ldots,p+n+q\}$. Thus, $t_i$ does not affect the labels of the elements in $D_\lambda$ for $i<p$ and $i>p+n+1$. Furthermore, since $f^{-1}(p)\in P$ and $f^{-1}(p+1)\in D_\lambda$, $p$ and $p+1$ always label comparable elements, so $t_p$ does nothing on $f$. Similarly, $t_{p+n}$ also does nothing. Hence, the only moves that affect the labels in $D_\lambda$ are $t_{p+1},\ldots,t_{p+n-1}$.
    \end{proof}

    \begin{lemma}\label{lem:q_k-labels}
        Evacuation $q_{k-1}$ affects the labels in $D_\lambda$ as follows.
        \begin{enumerate}
            \item If $k-1\leq p$, $I(q_{k-1}(f)) = I(f)$.
            \item If $p<k-1<p+n$, $I(q_{k-1}(f)) = q_{k-1-p}(I(f))$.
            \item If $p+n\leq k-1$, $I(q_{k-1}(f)) = q_{n-1}\partial_n^{k-p-n}(I(f))$.
        \end{enumerate}
    \end{lemma}

    \begin{proof}
        1) is clear since if $k-1\leq p$, $q_{k-1}$ only consists of $t_1,\ldots,t_p$, which do not affect the labels in $D_\lambda$.

        To prove 2), note that in this case, $q_{k-1} = \partial_1\ldots\partial_p\partial_{p+1}\ldots\partial_k$. Again, $\partial_1\ldots\partial_p$ only consist of $t_1,\ldots,t_p$, so they do not affect the labels in $D_\lambda$. On the other hand, in $\partial_{p+v} = t_{p+v-1}\ldots t_{p+1}t_p\ldots t_1$, the first $p$ terms $t_p\ldots t_1$ do not affect the labels in $D_\lambda$ either. Thus, the only terms in $\partial_{p+v}$ that affect the labels in $D_\lambda$ are $t_{p+v-1}\ldots t_{p+1}$, so the terms in $q_{k-1}$ that actually affect the labels in $D_\lambda$ are $t_{p+1}(t_{p+2}t_{p+1})\ldots(t_{k-1}\ldots t_{p+1})$. On $I(f)$, this is equivalent to $t_1(t_2t_1)\ldots(t_{k-1-p}\ldots t_1)$, which is $q_{k-1-p}$.

        Proving 3) is similar. Note that in this case, $q_{k-1} = \partial_1\ldots\partial_p \partial_{p+1}\ldots \partial_{p+n}\partial_{p+n+1}\ldots\partial_k$. Removing the terms that do not affect the labels in $D_\lambda$, we obtain
        \[ t_{p+1}(t_{p+2}t_{p+1})\ldots(t_{p+n-1}\ldots t_{p+1}) (t_{p+n-1}\ldots t_{p+1})^{k-p-n} \]
        where the first part $t_{p+1}(t_{p+2}t_{p+1})\ldots(t_{p+n-1}\ldots t_{p+1})$ comes from $\partial_{p+1}\ldots \partial_{p+n}$ and the second part $(t_{p+n-1}\ldots t_{p+1})^{k-p-n}$ comes from $\partial_{p+n+1}\ldots\partial_k$. On $I(f)$, the first part is equivalent to $t_1(t_2t_1)\ldots(t_{n-1}\ldots t_1)$, which is $q_{n-1}$, while the second part is equivalent to $(t_{n-1}\ldots t_1)^{k-p-n}$, which is $\partial_n^{k-p-n}$. This gives the desired product.
    \end{proof}

    Recall from Section \ref{sec:pro-eva} that we can associate each linear extension of $D_\lambda$ with an ordered set partition of $n$ numbers. Furthermore, $\partial_k$ is equivalent to left multiplication by $(k\ldots 21)$ while $q_{k-1}$ is equivalent to left multiplication by $\left(\begin{smallmatrix} 
    1 & 2 & \cdots & k\\
    k & k-1 & \cdots & 1
    \end{smallmatrix}\right)$. We claim that studying the action of promotion and evacuation on the ordered set partitions is enough to prove Proposition \ref{thm:main-thm}.

    \begin{lemma}\label{lem:neck-iff}
        $(t_iq_{jk})^2$ fixes the labels in $D_\lambda$ for any linear extension $f$ of $R$ if and only if it is equivalent to the identity permutation on ordered set partitions.
    \end{lemma}

    \begin{proof}
        If $(t_iq_{jk})^2$ is equivalent to the identity permutation on ordered set partitions, then for any linear extension $f$, let $L_m$ be the set of numbers labeling the $m$th chain of $D_\lambda$ in $I(f)$. Since $(t_iq_{jk})^2$ is the identity permutation, after $(t_iq_{jk})^2$, $L_m$ is the same set. $L_m$ uniquely determines the labeling of the $m$th chain, so $(t_iq_{jk})^2$ fixes $I(f)$. Thus, it fixes the labels of $D_\lambda$ in $f$.

        Conversely, if $(t_iq_{jk})^2$ is equivalent to $w \neq \text{id}$, then there exists a number $\ell$ such that $\ell \neq w(\ell)$. Since $D_\lambda$ has at least two chains, the ordered set partitions have at least two blocks. Thus, we can put $\ell$ and $w(\ell)$ in two different blocks. This gives a linear extension that is not fixed by $(t_iq_{jk})^2$ because the label $w(\ell)$ is in different chains before and after $(t_iq_{jk})^2$.
    \end{proof}
    
    From Lemma \ref{lem:q_k-labels}, we can easily deduce the effect of $q_{k-1}$ on the ordered set partition corresponding to $I(f)$.

    \begin{lemma}\label{lem:q_k-neck}
        Evacuation $q_{k-1}$ affects the ordered set partition corresponding to $I(f)$ as follows.
        \begin{enumerate}
            \item If $k-1\leq p$, $q_{k-1}$ does nothing to the ordered set partition.
            \item If $p<k-1<p+n$, $q_{k-1}$ is left multiplication by $\left(\begin{smallmatrix} 
            1 & 2 & \cdots & k-p\\
            k-p & k-p-1 & \cdots & 1
            \end{smallmatrix}\right)$.
            \item If $p+n\leq k-1$, $q_{k-1}$ is left multiplication by $\left(\begin{smallmatrix} 
            1 & 2 & \cdots & n\\
            n & n-1 & \cdots & 1
            \end{smallmatrix}\right)(n\ldots 21)^{k-p-n}$.
        \end{enumerate}
    \end{lemma}

    Thus, we have a nice description of the action of $q_{jk}$ for some value of $j<k$.

    \begin{corollary}\label{cor:q_jk-neck}
        Suppose $k = p+n+r$ for some $r\geq 0$ and $p < k-j < p+n$. Let $m = r~\text{mod}~n$ and $\ell = k-j-p$. Then $q_{jk}$ is left multiplication by
        \[ \left(\begin{smallmatrix} 
            m-\ell & m-\ell+1 & \cdots & m-1 & m\\
            m & m-1 & \cdots & m-\ell+1 & m-\ell
            \end{smallmatrix}\right) ,\]
        where all values in the matrix are taken mod $n$.
    \end{corollary}

    \begin{proof}
        By Lemma \ref{lem:q_k-labels}, $q_{jk} = q_{k-1}q_{k-j}q_{k-1}$ is equivalent to $q_{n-1}\partial_n^r q_{\ell} q_{n-1}\partial_n^r$ on $I(f)$. By Lemma \ref{lem:q_k-neck}, this is equivalent to left multiplication by
        \[ \left(\begin{smallmatrix} 
            1 & 2 & \cdots & n\\
            n & n-1 & \cdots & 1
            \end{smallmatrix}\right)(n\ldots 21)^{r} 
            \left(\begin{smallmatrix} 
            1 & 2 & \cdots & \ell+1\\
            \ell+1 & \ell+2 & \cdots & 1
            \end{smallmatrix}\right)
            \left(\begin{smallmatrix} 
            1 & 2 & \cdots & n\\
            n & n-1 & \cdots & 1
            \end{smallmatrix}\right)(n\ldots 21)^{r}. \]
        Since $(n\ldots 21)^n = 1$, we can simplify this to
        \[ \left(\begin{smallmatrix} 
            1 & 2 & \cdots & n\\
            n & n-1 & \cdots & 1
            \end{smallmatrix}\right)(n\ldots 21)^m 
            \left(\begin{smallmatrix} 
            1 & 2 & \cdots & \ell+1\\
            \ell+1 & \ell+2 & \cdots & 1
            \end{smallmatrix}\right)
            \left(\begin{smallmatrix} 
            1 & 2 & \cdots & n\\
            n & n-1 & \cdots & 1
            \end{smallmatrix}\right)(n\ldots 21)^m. \]
        Showing that this is equal to $\left(\begin{smallmatrix} 
            m-\ell & m-\ell+1 & \cdots & m-1 & m\\
            m & m-1 & \cdots & m-\ell+1 & m-\ell
            \end{smallmatrix}\right)$ is just a matter of keeping track of the permutation.

        \begin{center}
        
            \begin{tikzpicture}
                \node (p1) at (0,10) {$1,\ldots,m-\ell,\ldots,m,m+1,\ldots,n$};
                
                \node (p2) at (0,8) {$n-m+1,\ldots,n-\ell,\ldots,n,1,\ldots,n-m$};
                
                \node (p3) at (0,6) {$m,\ldots,\ell+1,\ldots,1,n,\ldots,m+1$};
                
                \node (p4) at (0,4) {$m,\ldots,1,\ldots,\ell+1,n,\ldots,m+1$};
                
                \node (p5) at (0,2) {$n,\ldots,n+1-m,\ldots,n+\ell+1-m,n-m,\ldots,1$};
                
                \node (p6) at (0,0) {$1,\ldots,m,\ldots,m-\ell,m+1,\ldots,n$};

                \draw[->] (p1)--(p2) node[midway,right] {%
                $(n\ldots 21)^m$};

                \draw[->] (p2)--(p3) node[midway,right] {%
                $\left(\begin{smallmatrix} 
                1 & 2 & \cdots & n\\
                n & n-1 & \cdots & 1
                \end{smallmatrix}\right)$};

                \draw[->] (p3)--(p4) node[midway,right] {%
                $\left(\begin{smallmatrix} 
                1 & 2 & \cdots & \ell+1\\
                \ell+1 & \ell+2 & \cdots & 1
                \end{smallmatrix}\right)$};

                \draw[->] (p4)--(p5) node[midway,right] {%
                $(n\ldots 21)^m$};

                \draw[->] (p5)--(p6) node[midway,right] {%
                $\left(\begin{smallmatrix} 
                1 & 2 & \cdots & n\\
                n & n-1 & \cdots & 1
                \end{smallmatrix}\right)$};
            \end{tikzpicture}
            
        \end{center}
    \end{proof}

    For example, the action of $q_{jk}$ on $D_{(4,2,2)}$ where $m = 2$ and $\ell = 5$ is left multiplication by $\left(\begin{smallmatrix} 
    6 & 7 & 8 & 1 & 2\\
    2 & 1 & 8 & 7 & 6
    \end{smallmatrix}\right)$.
    We can think of the action of $q_{jk}$ as ``flipping'' the interval between $m-\ell$ and $m$. On the other hand, $t_i$ is simply flipping $i-p$ and $i-p+1$. Recall that the condition $(t_iq_{jk})^2 = 1$ is equivalent to $t_i$ and $q_{jk}$ commuting. Since $t_i$ and $q_{jk}$ both flip intervals, if they flip disjoint intervals, they do commute. However, if the intervals they flip are not disjoint, most of the time they do not commute. This allows us to prove Proposition \ref{thm:main-thm}.

    \begin{lemma}\label{lem:bound}
        If $i\leq p$ or $k-j\leq p$, then $(t_iq_{jk})^2$ fixes the labels in $D_\lambda$.
    \end{lemma}

    \begin{proof}
        If $i\leq p$, then $t_i$ does not affect the labels in $D_\lambda$, so the only action that affects these labels is $q_{jk}^2$, which is the identity by Corollary \ref{cor:qjk-involution}.

        If $k-j\leq p$, then we claim that $q_{jk}$ fixes the labels in $D_\lambda$. Since $k-j\leq p$, only $q_{k-1}^2$ affects the labels in $D_\lambda$. Fortunately, $q_{k-1}^2$ is also the identity by Proposition \ref{prop:eva-involution}.
    \end{proof}

    Lemma \ref{lem:bound} implies condition 1 of Proposition \ref{thm:main-thm}.

    \begin{corollary}\label{cor:cond1}
        If $p>q+n-4$, then $(t_iq_{jk})^2$ fixes the labels in $D_\lambda$ for all $i+1<j<k$.
    \end{corollary}

    \begin{proof}
        If $p>q+n-4$, then $2p+4>p+n+q$. By Lemma \ref{lem:bound}, if $i\leq p$, then the statement is true. If $i\geq p+1$, then $j\geq p+3$, and since $k\leq |R| = p+n+q < 2p+4$, $k-j<p+1$. Hence, the statement is also true by Lemma \ref{lem:bound}.
    \end{proof}

    Now we prove condition 2 of Proposition \ref{thm:main-thm}.

    \begin{lemma}\label{lem:cond2}
        If $p=q+n-4$, then $(t_iq_{jk})^2$ fixes the labels in $D_\lambda$ for all $i+1<j<k$ if and only if $q~\text{mod}~n\neq 1,3$.
    \end{lemma}

    \begin{proof}
        The only tuple $(i,j,k)$ such that $i+1<j<k$ and $i,k-j>p$ is $(p+1,p+3,2p+4)$. For this tuple, $k-j = p+1$ and $(k-p-n)~\text{mod}~n = q~\text{mod}~n$. Thus, $t_i$ is equivalent to left multiplication by $(12)$ to the numbers on the ordered set partition, and $q_{jk}$ is equivalent to left multiplication by $(m-1,m)$ where $m = q~\text{mod}~n$ by Corollary \ref{cor:q_jk-neck}. Clearly, $(12)$ and $(m-1,m)$ commute if and only if $m\neq 1,3$.
    \end{proof}

    Finally, we prove condition 3 of Proposition \ref{thm:main-thm}.

    \begin{lemma}\label{lem:cond3}
        If $p=q+n-r$ for $r>4$, then $(t_iq_{jk})^2$ fixes the labels in $D_\lambda$ for all $i+1<j<k$ if and only if $q~\text{mod}~n > r-1$.
    \end{lemma}

    \begin{proof}
        Let $m = q~\text{mod}~n$. First, we consider $r = 5$. In this case, if $m = 1,3$, then for $(i,j,k) = (p+1,p+4,2p+5)$, $t_i$ and $q_{jk}$ do not commute by the same argument as in the proof of Lemma \ref{lem:cond2}. If $m = 2$, then for $(i,j,k) = (p+1,p+3,2p+5)$, $t_i$ is equivalent to left multiplication by $(12)$ to the ordered set partition, and $q_{jk}$ is equivalent to left multiplication by $\left(\begin{smallmatrix} 
        n & 1 & 2\\
        2 & 1 & n
        \end{smallmatrix}\right)$ by Corollary \ref{cor:q_jk-neck}. Clearly, these two do not commute. The same tuple and argument apply for the case $m = 4$ since $(12)$ and $\left(\begin{smallmatrix} 
        2 & 3 & 4\\
        4 & 3 & 2
        \end{smallmatrix}\right)$ also do not commute.

        Conversely, when $m > 4$, if $i = p+1$, the options for $(j,k)$ are $(p+3,2p+5)$, $(p+3,2p+4)$, and $(p+4,2p+5)$. In any case, $q_{jk}$ only affects $m, m-1, m-2 > 4-2 = 2$ while $t_i$ only affects $1$ and $2$. Thus, $t_i$ and $q_{jk}$ commute. If $i = p+2$, then the only option for $(j,k)$ is $(p+4,2p+5)$, so $q_{jk}$ only affects $m, m-1 > 4-1 = 3$ while $t_i$ only affects $2$ and $3$. Thus, $t_i$ and $q_{jk}$ also commute. This complete the argument for the case $r = 5$.

        The analogous argument applies when $r>5$. If $2<m\leq r-1$, then for $(i,j,k) = (p+1,p+r-m+2,2p+r)$, $t_i$ and $q_{jk}$ do not commute since $(12)$ and $\left(\begin{smallmatrix} 
        2 & 3 & \ldots & m\\
        m & m-1 & \ldots & 2
        \end{smallmatrix}\right)$ do not commute. Similarly, if $m=2$, for $(i,j,k) = (p+1,p+r-2,2p+r)$, $t_i$ and $q_{jk}$ do not commute since $(12)$ and $\left(\begin{smallmatrix} 
        n & 1 & 2\\
        2 & 1 & n
        \end{smallmatrix}\right)$ do not commute. If $m = 1$, for $(i,j,k) = (p+1,p+r-1,2p+r)$, $t_i$ and $q_{jk}$ do not commute since $(12)$ and $(1n)$ do not commute.

        Conversely, when $m > r-1$, if $i = p+\ell$, then $p+\ell+2\leq j < k \leq 2p+r$, so $k-j \leq p+r-\ell-2$. This means that $q_{jk}$ only affects $m,m-1,\ldots,m-r+\ell+2 > (r-1) - r + \ell +2 = \ell+1$ while $t_i$ only affects $\ell$ and $\ell+2$. Thus, $t_i$ and $q_{jk}$ commute. This completes the proof.
    \end{proof}

    Corollary \ref{cor:cond1}, Lemma \ref{lem:cond2}, and Lemma \ref{lem:cond3} together prove Proposition \ref{thm:main-thm}. Now in order to complete the characterization of LE-cactus posets in the family of ordinal sums of disjoint union of chains, we need the conditions for $\mathfrak{D}_1,\mathfrak{D}_2$, and $\mathfrak{D}_3$. The conditions for $\mathfrak{D}_3$ is quite similar to Proposition \ref{thm:main-thm}.

    \begin{prop}\label{prop:D3}
        $(p,3,q)$ is cactus-compatible if and only if
        \begin{enumerate}
            \item $p>q-1$, or
            \item $p = q-1$ and $q~\text{mod}~n\neq 1,3$.
        \end{enumerate}
        In particular, if $p\leq q-1$, $(p,3,q)$ is not cactus-compatible.
    \end{prop}

    The proof for Proposition \ref{prop:D3} is essentially the same as the proofs of Corollary \ref{cor:cond1} and Lemma \ref{lem:cond2}. The conditions for $\mathfrak{D}_1$ and $\mathfrak{D}_2$ are even simpler: there is no condition!

    \begin{prop}\label{prop:D1,2}
        For all $p,q$, $(p,1,q)$ and $(p,2,q)$ are cactus-compatible.
    \end{prop}

    \begin{proof}
        The case when $D_\lambda\in\mathfrak{D}_1$ is trivial since its element can only be labeled by $p+1$. When $D_\lambda\in\mathfrak{D}_2$, the only move that affects the labels in $D_\lambda$ is $t_{p+1}$. Fortunately, $t_{p+1}$ appears an even number of times in $(t_iq_{jk})^2$, so it fixes the labels in $D_\lambda$.
    \end{proof}

    Combining Proposition \ref{thm:main-thm}, Proposition \ref{prop:D3}, and Proposition \ref{prop:D1,2}, we achieve our desired characterization.

    \begin{theorem}\label{thm:char}
        Consider a sequence of positive integers $a_0=0,a_1,\ldots,a_\ell,a_{\ell+1} = 0$ and a sequence of posets $D_{\mu_1},\ldots,D_{\mu_\ell}$ where $D_{\mu_i}\in\mathfrak{D}_{a_i}$. The poset $P = D_{\mu_1}\oplus \ldots\oplus D_{\mu_\ell}$ is LE-cactus if and only if for all $i = 1,2,\ldots,\ell$, the triples
        \[ \left(\sum_{r=0}^{i-1} a_r, \,\, a_i, \,\, \sum_{r=i+1}^{\ell+1} a_r\right) \]
        are cactus-compatible.
    \end{theorem}

    The following corollary is immediate.

    \begin{corollary}\label{cor:D-cactus}
        For all $i$, all $D_\mu\in \mathfrak{D}_i$ are LE-cactus. 
    \end{corollary}

    \begin{remark}
        Corollary \ref{cor:D-cactus} can also be proved using Theorem \ref{thm:P+Q-cactus} from \cite{chiang2023bender} by observing that every chain is LE-cactus, and hence their disjoint unions are also LE-cactus.
    \end{remark}

\section{Discussion}\label{sec:discussion}

    Unfortunately, Theorem \ref{thm:char} does not apply to other posets since for an arbitrary poset, it is not easy to understand the action of promotion and evacuation. Specifically, if a connected component is not a chain, then the set of labels in that component does not uniquely determine the labeling. However, we do have a {\it necessary} condition for a poset to be LE-cactus.

    \begin{prop}\label{prop:checking}
        Let $D$ be a disconnected poset with $|D| = n$, and let $R = P~\oplus~D~\oplus~Q$ for posets $P$ and $Q$ with $|P| = p$ and $|Q| = q$. Then $R$ is LE-cactus only if $(p,n,q)$ satisfies the conditions in Proposition \ref{thm:main-thm} and \ref{prop:D3}.
    \end{prop}

    \begin{proof}
        Since $D$ is disconnected, we can write $D$ as a disjoint union $D_1 \oplus \ldots \oplus D_\ell$. Since $R$ is LE-cactus, all elements $(t_i q_{jk})^2$ fix the labels in $D$. In particular, the labels in $D_i$ stay in $D_i$ under the action of these elements. By the same ordered set partition argument as in Section \ref{sec:main-thm}, this happens only if $(p,n,q)$ satisfies the conditions in Proposition \ref{thm:main-thm} and \ref{prop:D3}.
    \end{proof}

    Furthermore, in \cite[Proposition 3.18, 3.19, 3.20]{chiang2023bender}, it was proved that if $P$ is LE-cactus then $A_1 \oplus P$ and $A_2 \oplus P$ are LE-cactus, but $A_i\oplus P$ is not for any $i \geq 3$, where $A_i$ is an antichain of $i$ elements. This can be seen from our main theorem: since $A_i\in \mathfrak{D_i}$, and $A_i \oplus P$ does not satisfy the conditions in Proposition \ref{thm:main-thm} and \ref{prop:D3}, the labels in $A_i$ are not fixed by all elements $(t_i q_{jk})^2$. However, these elements $(t_i q_{jk})^2$ do fix the labels in $P$. In general, we have the following proposition.

    \begin{prop}
        Let $P$ be an LE-cactus poset. Consider a sequence of positive integers $a_0=0,a_1,\ldots,a_\ell,a_{\ell+1} = |P|$ and a sequence of posets $D_{\mu_1},\ldots,D_{\mu_\ell}$ where $D_{\mu_i}\in\mathfrak{D}_{a_i}$. The poset $R = D_{\mu_1}\oplus \ldots\oplus D_{\mu_\ell}\oplus P$ is LE-cactus if and only if for all $i = 1,2,\ldots,\ell$, the triples
        \[ \left(\sum_{r=0}^{i-1} a_r, \,\, a_i, \,\, \sum_{r=i+1}^{\ell+1} a_r\right) \]
        are cactus-compatible.
    \end{prop}

    \begin{proof}
        We already know that for $r=1,2,\ldots,\ell$, the elements $(t_i q_{jk})^2$ fixed the labels in $D_{\mu_r}$ if and only if $a_r$ satisfies the conditions in Proposition \ref{thm:main-thm} and \ref{prop:D3}. Thus, it suffices to prove that the labels in $P$ are always fixed by the elements $(t_i q_{jk})^2$.

        Let $m = a_1+\ldots+a_\ell$, by similar reasoning to the arguments above, if $i\leq m$ or $k-j\leq m$, then $(t_iq_{jk})^2$ fixes the labels in $P$. If $i, k-j > m$, then on the labels in $P$, $t_i$ is equivalent to $t_{i-m}$ and $q_{jk}$ is equivalent to $q_{k-1-m}q_{k-j-m}q_{k-1-m}$, which is $q_{k-m,j}$. Thus, $(t_iq_{jk})^2$ is equivalent to $(t_{i-m}q_{k-m,j})^2$. Since $i+1<j$, $i-m+1<j$, and since $k-j>m$, $k-m>j$. Thus we do have the condition $i-m+1<j<k-m$, so this element $(t_{i-m}q_{k-m,j})^2$ is actually acting on linear extensions of $P$. Since $P$ is LE-cactus, this element fixes the labels in $P$.
    \end{proof}

    Therefore, this gives hope for an analogue of Proposition \ref{thm:main-thm} for LE-cactus poset: Let $D$ be an LE-cactus poset, and $R = P \oplus D \oplus Q$ for some finite posets $P$ and $Q$. Are there any conditions on $|P|, |D|, |Q|$ that characterize when the labels in $D$ are fixed by all elements $(t_i q_{ij})^2)$? Corollary \ref{cor:D-cactus} says that all disjoint union of chains are LE-cactus, so we expect an answer to this question to be a generalization of Proposition \ref{thm:main-thm}. This would give a clearer insight into the behavior of LE-cactus posets under the ordinal sum operation.


\bibliography{bibliography}
\bibliographystyle{alpha}

\end{document}